\documentclass[12pt,a4paper]{article} 
\usepackage{amsmath}
\usepackage{amsthm}
\usepackage{amsfonts}
\usepackage{amssymb}
\usepackage{bussproofs}
\author{Amir Akbar Tabatabai}
\theoremstyle{plain} 
\newtheorem{thm}{Theorem}[section]

\newtheorem{lem}[thm]{Lemma}

\theoremstyle{definition}
\newtheorem{dfn}[thm]{Definition}
\newtheorem{exam}[thm]{Example}
\newtheorem{rem}[thm]{Remark}

\def\PA{\mathrm{PA}}
\def\Pr{\mathrm{Pr}}

\def\S4{\mathrm{S4}}
\def\Cons{\mathrm{Cons}}
\def\Rfn{\mathrm{Rfn}}

\begin{document}
\title{Russellian Propositional Logics and the BHK Interpretation} 

\author{Amirhossein Akbar Tabatabai \footnote{The author is supported by the ERC Advanced Grant 339691 (FEALORA)}\\
Institute of Mathematics\\
Academy of Sciences of the Czech Republic\\
tabatabai@math.cas.cz}

\date{February 9, 2017}

\maketitle

\begin{abstract}
The BHK interpretation interprets propositional statements as descriptions of the world of proofs; a world which is hierarchical in nature. It consists of different layers of the concept of proof; the proofs, the proofs about ``proofs" and so on. To describe this hierarchical world, one approach is the Russellian approach in which we use a typed language to reflect this hierarchical nature in the syntax level. In this case, since the connective responsible for this hierarchical behavior is implication, we will use a \textit{typed} language equipped with a hierarchy of implications, $\{\rightarrow_n\}_{n=0}^{\infty}$. In fact, using this typed propositional language, we will introduce the hierarchical counterparts of the logics $\mathbf{BPC}$, $\mathbf{EBPC}$, $\mathbf{IPC}$ and $\mathbf{FPL}$ and then by proving their corresponding soundness-completeness theorems with respect to their natural BHK interpretations, we will show how these different logics describe different worlds of proofs embodying  different hierarchical behaviors.
\end{abstract}

\section{Introduction}
The intuitionistic tradition is based on the core concept of proof and its most important claim is that the statements in the language are descriptions of the world of proofs rather than the actual Platonistic world.\footnote{Notice that our paraphrase of the intuitionistic tradition's claim seems extremly unorthodox. The reason is that we refer to the Platonistic world which can not exist in the intuitionistic terms. But it is not actually a paradox or a misinterpretation. In fact, following G\"{o}del, \cite{G}, we interpret intuitionism from outside of its paradigm and we assume that we live in a Platonistic world in which intuitionism is a way of describing the sub-world of the \textit{classical proofs}. This approach has been shown to be extremely useful. For more discussion about this approach and other approaches, see \cite{AK1}.} Therefore, the intended semantics of intuitionistic mathematics is just a relation which describes the way that intuitionistic logic describes the world of proofs. This semantics has a very complicated character and resists any kind of natural formalization. The reason is the hierarchical nature of the concept of proof which claims that there are different layers in the world of proofs. The first layer is the layer of proofs about the facts of the world. The second is the level of proofs about the proofs in the first level and there are also the third one, the fourth and so on. The important thing is that these higher level proofs occur in a very natural way. For instance, consider the statement ``\textit{$A$ implies $B$}". A proof for this statement is a \textit{proof} which transforms proofs of $A$ to proofs for $B$. Hence, this \textit{proof} is about proofs, meaning that it belongs to a higher level of the levels of the proofs of $A$ and $B$. \\

So far we have shown that the world of proofs is hierarchical in nature. Therefore, like any other hierarchical phenomenon, there are two main approaches to describe it: The Zermelo approach and the Russelian one. The Zermelo approach uses an \textit{untyped} language to investigate the \textit{orderless propositions} about the phenomenon, i.e. the propositions which are true in the world regardless of the levels of the objects in the world. The classical example of this approach is Zermelo's set theory in the untyped first order language. Although there is a canonical order in the world of sets, i.e. the rank of sets, Zermelo's set theory ignores this order and axiomatizes the universe uniformly. \\

Now, consider using the Zermelo approach for the world of proofs. We need two different components. First, an untyped language which in this case is the usual propositional language and second, an interpretation to interpret this language as a way of describing the world of proofs. The latter is provided by the following BHK interpretation:\\
\\
$\bullet$ a proof for $A \wedge B$ is a pair of a proof for $A$ and a proof for $B$.\\
$\bullet$ a proof for $A \vee B$ is a proof for $A$ or a proof for $B$.\\
$\bullet$ a proof for $A \rightarrow B$ is a construction which transforms any proof of $A$ to a proof for $B$.\\
$\bullet$ $\bot$ does not have any proof.\\

This approach seems natural and easy to follow. We have an untyped simple language which is informally capable of describing the whole complexity of the world of proofs. Therefore, it seems natural to try to formalize this BHK interpretation as a formalized intended semantics for intuitionistic mathematics. But there are some fundamental problems along the way. The reason is that some of the most important statements in the language do have an internal inherent order and therefore the Zermelo approach has to ignore them; the situation which is not what we expect. Let us illuminate the idea by an example. Consider the following theorem of $\mathbf{IPC}$: $(\top \rightarrow \bot) \rightarrow \bot$. If we interpret this statement by the BHK interpretation, its content would be the following: ``There is a \textbf{proof} which shows that the \textit{provability} of the provability of $\bot$ implies the provability of $\bot$." It seems that just by the informal interpretation of this statement we can be sure of its truth and it should be considered as an axiom in our Zermelo theory. The reason is that it is just a special case of the soundness of our theories which we intuitively believe. On the other hand, investigating the formula more precisely, we will notice that there are different levels of proofs in the statement, and if we can ignore the order in the world of proofs, then it should be true or false regardless the levels of proofs we are using. If we interpret all of the proofs as proofs in the same level, say in the theory $T$, then the statement means that $\Box_T(\Box_T \Box_T \bot \rightarrow \Box_T \bot)$. Using L\"{o}b's theorem, it implies that $\Box_T \Box_T \bot$ which contradicts with the intuitive condition that there are no proofs of $\bot$. But if we interpret the ``provability" in the statement as the provability in $T$, the \textit{provability} as the provability in its meta-theory $S$ and \textbf{proof} as a proof in the meta-theory of $S$, say $R$, then it means that $R$ proves $\Box_S \Box_T \bot \rightarrow \Box_T \bot$. This is possible if $R$ is strong enough to prove the soundness of $S$; a condition that seems natural and acceptable. It actually is the informal reason behind our belief in the intuitive truth of this theorem of $\mathbf{IPC}$.\\
This observation shows that we can not ignore the natural order of implications in the statements. Hence, we can not follow the Zermelo approach in a very natural way of ignoring orders in interpreting the implications. But is there any way to deal with these kinds of referring problems in the Zermelo approach? Can we handle it? In \cite{AK1} we showed that by some natural techniques, it is possible to use an untyped language to describe the hierarchical world of proofs. However, we have to emphasize that those techniques make the whole work too complicated to follow. For more information, see \cite{AK1}.\\

Another approach is the Russellian approach. In spite of the Zermelo approach, it is based on using a \textit{typed} language to reflect the hierarchical nature of the world in a very syntactical way. Here again, the main example is set theory, in this case, the Russellian one. As it is well-known, it limits the syntax of the language in a way that we can not use statements such as $a \in b$ in which the level of $b$ is less than or equal to the level of $a$. Now, consider using this approach for the world of proofs. It is clear that instead of just one implication, we need a hierarchy of implications, $\{\rightarrow_n\}_{n=0}^{\infty}$ to reflect the hierarchical concept of provability. Moreover, we have to limit the formulas in the language in a way that $A \rightarrow_n B$ is a formula iff $n$ is larger than all the indices of the implications in $A$ and $B$. It means that the implication $\rightarrow_n$ refers to a higher level than the implications in $A$ or $B$. In the following, we will persue the Russellian approach. Indeed, we will use the above mentioned language to formalize the hierarchical world of proofs on the one hand, and define the BHK interpretation as the intended interpretation of this language on the other. Then, we will establish the soundness-completeness results for some of the propositional logics with respect to their canonical BHK interpretations. These logics can be considered as the description of different kinds of behaviors of the hierarchies of provabilities. 
\section{Preliminaries}
Our main strategy to prove the soundness-completeness theorems for our propositional logics is reducing them to the soundness-completeness theorems for their modal counterparts introduced in \cite{AK2}. In this section we will cover what we need from \cite{AK2} to follow this strategy.\\

First of all we need a new modal language to capture the provability-based behavior of hierarchies: 
\begin{dfn}\label{t2-1}
Consider the language of modal logics with infinitely many modalities, $\{\Box_n\}_{n=0}^{\infty}$, The set of formulas in this language, $\mathcal{L}_{\infty}^{\Box}$, is defined as the least set of expressions which includes all atomic formulas and is closed under all boolean operations and also the following operation: If $A \in \mathcal{L}_{\infty}^{\Box}$ and $n$ is bigger than all indices of boxes occurred in $A$ then $\Box_n A \in \mathcal{L}_{\infty}^{\Box}$. In other words, $A \in \mathcal{L}_{\infty}^{\Box}$, if $A$ is a usual formula in the modal language and also the index of any box is bigger than the indices of all other boxes in its scope.
\end{dfn}
The intuition behind this definition is that the outer box refers to the provability predicate of a meta-theory and the inner boxes refer to the lower theories in the hierarchy. Therefore, it seems natural to assume that the situation in which a theory speaks about itself or higher theories, should be considered as a syntactical error. \\

After introducing the language, we need the intended semantics. The following is a formalization of the combination of a real world which atomic statements informally refer to, and the hierarchy of theories, meta-theories, meta-meta-theories and so on to interpret the boxes in the language.
\begin{dfn}\label{t2-2} 
A provability model is a pair $(M, \{T_n\}_{n=0}^{\infty})$ where $M$ is a model of $I\Sigma_1$ and $\{T_n\}_{n=0}^{\infty}$ is a hierarchy of arithmetical r.e. theories such that for any $n$, $I\Sigma_1 \subseteq T_n \subseteq T_{n+1}$ provably in $I\Sigma_1$.
\end{dfn}
We also need the notions of arithmetical substitution, evaluation of modal formulas by substitutions and finally the satisfaction relation:
\begin{dfn}\label{t2-3}
Let $(M, \{T_n\}_{n=0}^{\infty})$ be a provability model and $A \in \mathcal{L}_{\infty}^{\Box}$ be a formula. Then by an arithmetical substitution $\sigma$, we mean a function from atomic formulas to the set of arithmetical sentences. Moreover, by $A^{\sigma}$ we mean an arithmetical sentence which is resulted by substituting the atomic formulas by $\sigma$ and interpreting any $\Box_n$ as the provability predicate of $T_n$. The interpretation of boolean connectives are themselves. Moreover, if $\Gamma$ is a sequence of formulas $A_i$, by $\Gamma^{\sigma}$ we mean the sequence of $A_i^{\sigma}$.
\end{dfn}
\begin{dfn}\label{t2-4}
Let $(M, \{T_n\}_{n=0}^{\infty})$ be a provability model and $A \in \mathcal{L}_{\infty}^{\Box}$ be a formula. Then we say $(M, \{T_n\}_{n=0}^{\infty}) \vDash A$ if for any arithmetical substitution $\sigma$, $M \vDash A^{\sigma}$. Moreover, if $\Gamma$ and $\Delta$ are sequences of formulas and $\mathcal{C}$ a class of provability models, by $\mathcal{C} \vDash \Gamma \Rightarrow \Delta$, we mean that for any $(M, \{T_n\}_{n=0}^{\infty}) \in \mathcal{C}$, and for any arithmetical substitution $\sigma$, if $M \vDash \bigwedge \Gamma^{\sigma}$, then $M \vDash \bigvee \Delta^{\sigma}$.
\end{dfn}
So far, we have defined provability models to capture the informal hierarchical concept of provability and also we introduced an appropriate language to reflect these models. Next is a definition of some of the natural modal theories in this language to formalize some of the natural properties of hierarchies: 
\begin{dfn}\label{t2-5}
Consider the following set of axioms:
\begin{itemize}
\item[$(\mathbf{H})$]
$\Box_{n} A \rightarrow \Box_{n+1} A$
\item[$(\mathbf{K}_h)$]
$\Box_n (A \rightarrow B) \rightarrow (\Box_n A \rightarrow \Box_n B)$
\item[$(\mathbf{4}_h)$]
$\Box_n A \rightarrow \Box_{n+1} \Box_n A$
\item[$(\mathbf{D}_h)$]
$\neg \Box_n \bot $
\item[$(\mathbf{L}_h)$]
$\Box_{n+1}(\Box_n A \rightarrow A) \rightarrow \Box_n A$
\item[$(\mathbf{T}_h)$]
$\Box_n A \rightarrow A$
\item[$(\mathbf{5}_h)$]
$\neg \Box_n A \rightarrow \Box_{n+1} \neg \Box_n A$
\end{itemize}
Let $X$ be a set of these axioms. By $L(X)$ we mean the least set of formulas in $\mathcal{L}_{\infty}^{\Box}$, which contains all classical tautologies on formulas in $\mathcal{L}_{\infty}^{\Box}$, includes all instances of the set $X$ and is closed under the following rules:
\begin{itemize}
\item[$(\mathbf{MP})$]
If $A \in L(X)$ and $A \rightarrow B \in L(X)$ then $B \in L(X)$.
\item[$(\mathbf{NC}_h)$]
If $A \in L(X)$ then $\Box_n A \in L(X)$.
\end{itemize}
Moreover, if $\Gamma \cup \{A\} \subseteq \mathcal{L}_{\infty}$, by $\Gamma \vdash_{L(X)} A$ we mean that there exists a finite set $\Delta \subseteq \Gamma$ such that $L(X) \vdash \bigwedge \Delta \rightarrow A$.\\
Finally, we define $\mathbf{K4}_h=L(\mathbf{H}, \mathbf{K}_h, \mathbf{4}_h)$, $\mathbf{KD4}_h=L(\mathbf{H}, \mathbf{K}_h, \mathbf{4}_h, \mathbf{D}_h)$, $\mathbf{S4}_h=L(\mathbf{H}, \mathbf{K}_h, \mathbf{4}_h, \mathbf{T}_h)$, $\mathbf{GL}_h=L(\mathbf{H}, \mathbf{K}_h, \mathbf{4}_h, \mathbf{L}_h)$, $\mathbf{KD45}_h=L(\mathbf{H}, \mathbf{K}_h, \mathbf{D}_h, \mathbf{4}_h, \mathbf{5}_h)$ and $\mathbf{S5}_h=L(\mathbf{H}, \mathbf{K}_h, \mathbf{4}_h, \mathbf{T}_h, \mathbf{5}_h)$.
\end{dfn}
Fortunately, some of these logics have a nice proof theoretic behavior. For instance, the logics $\mathbf{K4}_h$, $\mathbf{KD4}_h$ and $\mathbf{S4}_h$ have reasonable sequent calculi which have the cut elimination property. To introduce them, consider the following set of rules: 
\begin{flushleft}
	\textbf{Axioms:}
\end{flushleft}
\begin{center}
	\begin{tabular}{c c}
		\AxiomC{$  A \Rightarrow A$ }
		\DisplayProof 
			&
		\AxiomC{$ \bot \Rightarrow  $}
		\DisplayProof
	\end{tabular}

\end{center}
\begin{flushleft}
 		\textbf{Structural Rules:}
\end{flushleft}
\begin{center}
	\begin{tabular}{c}
		\begin{tabular}{c c}
		\AxiomC{$\Gamma  \Rightarrow \Delta$}
		\LeftLabel{\tiny{$ (wL) $}}
		\UnaryInfC{$\Gamma,  A  \Rightarrow \Delta$}
		\DisplayProof
			&
		\AxiomC{$\Gamma  \Rightarrow \Delta$}
		\LeftLabel{\tiny{$ ( wR) $}}
		\UnaryInfC{$\Gamma \Rightarrow  \Delta, A$}
		\DisplayProof
		\end{tabular}
			\\[3 ex]
			\begin{tabular}{c c}
		\AxiomC{$\Gamma, A, A \Rightarrow \Delta$}
		\LeftLabel{\tiny{$ (cL) $}}
		\UnaryInfC{$\Gamma,  A  \Rightarrow \Delta$}
		\DisplayProof
		    &
		\AxiomC{$\Gamma \Rightarrow \Delta, A, A$}
		\LeftLabel{\tiny{$ (cR) $}}
		\UnaryInfC{$\Gamma \Rightarrow \Delta, A$}
		\DisplayProof
		\end{tabular}
        \\[3 ex]
	    
	    \AxiomC{$\Gamma_0 \Rightarrow \Delta_0, A$}
	    \AxiomC{$\Gamma_1, A \Rightarrow \Delta_1$}
		\LeftLabel{\tiny{$ (cut) $}}
		\BinaryInfC{$\Gamma_0, \Gamma_1 \Rightarrow \Delta_0, \Delta_1$}
		\DisplayProof
	\end{tabular}
\end{center}		
\begin{flushleft}
  		\textbf{Propositional Rules:}
\end{flushleft}
\begin{center}
  	\begin{tabular}{c c}
  		\AxiomC{$\Gamma_0, A  \Rightarrow \Delta_0 $}
  		\AxiomC{$\Gamma_1, B  \Rightarrow \Delta_1$}
  		\LeftLabel{{\tiny $\vee L$}} 
  		\BinaryInfC{$ \Gamma_0, \Gamma_1, A \lor B \Rightarrow \Delta_0, \Delta_1 $}
  		\DisplayProof
	  		&
	   	\AxiomC{$ \Gamma \Rightarrow \Delta, A_i$}
   		\RightLabel{{\tiny $ (i=0, 1) $}}
   		\LeftLabel{{\tiny $\vee R$}} 
   		\UnaryInfC{$ \Gamma \Rightarrow \Delta, A_0 \lor A_1$}
   		\DisplayProof
	   		\\[3 ex]
   		\AxiomC{$ \Gamma, A_i \Rightarrow \Delta$}
   		\RightLabel{{\tiny $ (i=0, 1) $}} 
   		\LeftLabel{{\tiny $\wedge L$}}  		
   		\UnaryInfC{$ \Gamma, A_0 \land A_1 \Rightarrow \Delta, C $}
   		\DisplayProof
	   		&
   		\AxiomC{$\Gamma_0  \Rightarrow \Delta_0, A$}
   		\AxiomC{$\Gamma_1  \Rightarrow  \Delta_1, B$}
   		\LeftLabel{{\tiny $\wedge R$}} 
   		\BinaryInfC{$ \Gamma_0, \Gamma_1 \Rightarrow \Delta_0, \Delta_1, A \land B $}
   		\DisplayProof
   			\\[3 ex]
   		\AxiomC{$ \Gamma_0 \Rightarrow A, \Delta_0 $}
  		\AxiomC{$ \Gamma_1, B \Rightarrow \Delta_1, C $}
  		\LeftLabel{{\tiny $\rightarrow L$}} 
   		\BinaryInfC{$ \Gamma_0, \Gamma_1, A \rightarrow B \Rightarrow \Delta_0, \Delta_1, C$}
   		\DisplayProof
   			&
   		\AxiomC{$ \Gamma, A \Rightarrow B, \Delta $}
   		\LeftLabel{{\tiny $\rightarrow R$}} 
   		\UnaryInfC{$ \Gamma \Rightarrow \Delta, A \rightarrow B$}
   		\DisplayProof
   		\\[3 ex]
   		\AxiomC{$ \Gamma \Rightarrow \Delta, A $}
   		\LeftLabel{\tiny {$\neg L$}} 
   		\UnaryInfC{$ \Gamma, \neg A \Rightarrow \Delta$}
   		\DisplayProof
   			&
   		\AxiomC{$ \Gamma, A \Rightarrow \Delta $}
   		\LeftLabel{{\tiny $\neg R$}} 
   		\UnaryInfC{$ \Gamma \Rightarrow \Delta, \neg A$}
   		\DisplayProof
	\end{tabular}
\end{center}
\begin{flushleft}
\textbf{Modal Rules:}
\end{flushleft}
\begin{center}
  	\begin{tabular}{c c c}
		\AxiomC{$ \{\sigma_r\}_{r \in R}, \{ \gamma_i, \Box_{n_i} \gamma_i\}_{i \in I} \Rightarrow A$}
		\LeftLabel{\tiny{$\Box_{4_h} R$}}
		\UnaryInfC{$\{\Box_n \sigma_r\}_{r \in R}, \{ \Box_{n_i} \gamma_i\}_{i \in I} \Rightarrow \Box_n A$}
		\DisplayProof
		&
		\AxiomC{$ \{\sigma_r\}_{r \in R}, \{ \gamma_i, \Box_{n_i} \gamma_i\}_{i \in I} \Rightarrow $}
		\LeftLabel{\tiny{$\Box_{D_h} R$}}
		\UnaryInfC{$\{\Box_n \sigma_r\}_{r \in R}, \{\Box_{n_i} \gamma_i\}_{i \in I} \Rightarrow $}
		\DisplayProof
		\\[2 ex]
		\end{tabular}
\end{center}
\begin{center}
  	\begin{tabular}{c c}	
        \AxiomC{$ \Gamma, A \Rightarrow \Delta$}
		\LeftLabel{\tiny{$\Box_h L$}}
		\UnaryInfC{$\Gamma, \Box_n A \Rightarrow \Delta$}
		\DisplayProof
		&
		\AxiomC{$ \{\sigma_r\}_{r \in R}, \{ \Box_{n_i} \gamma_i\}_{i \in I} \Rightarrow A$}
		\LeftLabel{\tiny{$\Box_{S_h} R$}}
		\UnaryInfC{$\{\Box_n \sigma_r\}_{r \in R}, \{ \Box_{n_i} \gamma_i\}_{i \in I} \Rightarrow \Box_n A$}
		\DisplayProof
		\\[2 ex]
	\end{tabular}
\end{center}
The condition of applying the rules $\Box_{4_h} R$, $\Box_{D_h} R$ and $\Box_{S_h} R$ is that for all $i \in I$, $n_i < n$.\\
The system $G(\mathbf{K4}_h)$ is the system that consists of the axioms, structural rules, propositional rules and the modal rule $\Box_{4_h} R$. $G(\mathbf{KD4}_h)$ is $G(\mathbf{K4}_h)$ plus the rule $\Box_{D_h} R$ and finally, $G(\mathbf{S4}_h)$ is the system $G(\mathbf{K4}_h)$ when we replace the rule $\Box_{4_h} R$ by $\Box_{S_h} R$ and add the rule $\Box_{h} L$.\\
\begin{thm}\label{t2-6} \cite{AK2} 
The systems $G(\mathbf{K4}_h)$, $G(\mathbf{KD4}_h)$ and $G(\mathbf{S4}_h)$ are equivalent to the logics $\mathbf{K4}_h$, $\mathbf{KD4}_h$ and $\mathbf{S4}_h$, respectively. Moreover, all of these sequent calculi have the cut elimination property.
\end{thm}
Using these sequent calculi we also proved the strong disjunction property.
\begin{dfn}\label{t2-7}
Logic $L$ has the strong disjunction property if for all formulas $\Box_n A$ and $\Box_m B$, if $L \vdash \Box_n A \vee \Box_m B$ then $L \vdash A$ or $L \vdash B$.
\end{dfn}
\begin{thm}\label{t2-8} (Strong disjunction property) \cite{AK2} 
All of the logics $\mathbf{K4}_h$, $\mathbf{KD4}_h$, $\mathbf{S4}_h$ and $\mathbf{GL}_h$ have strong disjunction property.
\end{thm}
And finally, the natural classes of provability models and the soundness-completeness theorem:
\begin{dfn}\label{t2-9}
\begin{itemize}
\item[$(i)$]
The class of all provability models will be denoted by $\mathbf{PrM}$.
\item[$(ii)$]
A provability model $(M,\{T_n\}_{n=0}^{\infty})$ is called consistent if for any $n$, $M$ thinks that $T_n$ is consistent and $T_{n+1} \vdash \Cons(T_n)$, i.e. $M \vDash \Cons(T_n)$ and $M \vDash \Pr_{T_{n+1}}(\Cons(T_n))$. Moreover, the class of all consistent provability models will be denoted by $\mathbf{Cons}$.
\item[$(iii)$]
A provability model $(M,\{T_n\}_{n=0}^{\infty})$ is reflexive if for any $n$, $M$ thinks that $T_n$ is sound and $T_{n+1} \vdash \Rfn(T_n)$, i.e. $M \vDash \Pr_{T_n}(A) \rightarrow A$ and $M \vDash \Pr_{T_{n+1}}(\Pr_{T_n}(A) \rightarrow A)$ for any sentence $A$. Moreover, the class of all reflexive provability models will be denoted by $\mathbf{Ref}$.
\item[$(iv)$]
A provability model $(M, \{T_n\}_{n=0}^{\infty})$ is constant if for any $n$ and $m$, $(M, \{T_n\}_{n=0}^{\infty})$ thinks that $T_n=T_m$, i.e. $M \vDash \Pr_{T_m}(A) \leftrightarrow \Pr_{T_n}(A)$ and $M \vDash \Pr_{T_0}(\Pr_{T_m}(A) \leftrightarrow \Pr_{T_n}(A))$ for any sentence $A$. The class of all constant provability models will be denoted by $\mathbf{Cst}$.
\end{itemize}
\end{dfn}

\begin{thm}\label{t2-10}(Soundness-Completeness)\cite{AK2}
\begin{itemize}
\item[$(i)$]
$\Gamma \vdash_{\mathbf{K4}_h} A$ iff $\mathbf{PrM} \vDash \Gamma \Rightarrow A$.
\item[$(ii)$]
$\Gamma \vdash_{\mathbf{KD4}_h} A$ iff $\mathbf{Cons} \vDash \Gamma \Rightarrow A$. 
\item[$(iii)$]
$\Gamma \vdash_{\mathbf{S4}_h} A$ iff $\mathbf{Ref} \vDash \Gamma \Rightarrow A$.
\item[$(iv)$]
$\Gamma \vdash_{\mathbf{GL}_h} A$ iff $\mathbf{Cst} \vDash \Gamma \Rightarrow A$.
\end{itemize}
\end{thm}
\section{Russellian Propositional Logics}
In this section we will define an appropriate language to capture the hierarchical nature of intuitionism and its BHK interpretation. Then we will define some natural theories in this language and finally we will use the G\"{o}del translation to find a connection between these propositional logics and the modal systems introduced in the Preliminaries.
\begin{dfn}\label{t3-1}
Consider the language of propositional logics in which the implication is replaced by infinitely many implications, $\{\rightarrow_n\}_{n=1}^{\infty}$. Define the set of formulas, $\mathcal{L}_{\infty}$, as the least set of expressions that includes all atomic variables, $\top$ and $\bot$, closed under conjunction and disjunction and finally closed under the following operation: If $A, B \in \mathcal{L}_{\infty}$, and $n$ is strictly greater than all numbers occurring as indices of implications in $A$ and $B$, then $A \rightarrow_n B \in \mathcal{L}_{\infty}$.  
\end{dfn}
\begin{rem}\label{t3-2}
For the simplicity, negations are not assumed as primitives in the language. But for any type of implication, we can define $\neg_n A$ as $A \rightarrow_n \bot$.
\end{rem}
Just like the modal language introduced in the Preliminaries, we assume that the indices of implications should be in an increasing order. The reason again is that we think that the situation in which a theory speaks about itself or the higher levels in the hierarchy should be considered as a syntactical error. This is actually the essence of the Russellian approach discussed in the Introduction. To have an example, notice that the expression $((p \rightarrow_1 q)\wedge r) \rightarrow_2 s $ is a formula in the language $\mathcal{L}_{\infty}$, while the expression $(p \rightarrow_1 (q \rightarrow_1 r))$ is not. The reason is that in the second expression the theory $T_1$ speaks about the provability behavior of itself which is not valid.\\

To introduce some formal systems in this language, consider the following set of natural deduction rules:

\begin{flushleft}
   \textbf{Propositional Rules:}
   \end{flushleft}
\begin{center}
  	\begin{tabular}{c c}
		\AxiomC{$ A$}
	   	\AxiomC{$ B$}
	   	\LeftLabel{\tiny{$\wedge I$}}
	   	\BinaryInfC{$A \wedge B $}
	   	\DisplayProof
		&
		\AxiomC{$ A \wedge B $}
		\LeftLabel{\tiny{$\wedge E$}}
		\UnaryInfC{$ A$}
		\DisplayProof
			
		\AxiomC{$ A \wedge B $}
		\LeftLabel{\tiny{$\wedge E$}}
		\UnaryInfC{$ B$}
		\DisplayProof
	   		\\[4 ex]
	   		
        \AxiomC{$  A $}
        \LeftLabel{\tiny{$\vee I$}}
   		\UnaryInfC{$ A \vee B$}
   		\DisplayProof
   		
   		\AxiomC{$  B $}
   		\LeftLabel{\tiny{$\vee I$}}
   		\UnaryInfC{$ A \vee B$}
   		\DisplayProof
			&
   		\AxiomC{$A \lor B$}
        \AxiomC{[$A$]}
        \noLine
   		\UnaryInfC{$\mathcal{D}$}
        \noLine
        \UnaryInfC{$C$}
        
        \AxiomC{[$B$]}
        \noLine
   		\UnaryInfC{$\mathcal{D'}$}
        \noLine
        \UnaryInfC{$C$}
        \LeftLabel{\tiny{$\vee E$}}
        \TrinaryInfC{$C$}
        \DisplayProof
    		\\[4 ex]
    		
        \AxiomC{[$A$]}
   		\noLine
   		\UnaryInfC{$\mathcal{D}$}
   		\noLine
   		\UnaryInfC{$B$}
   		\LeftLabel{\tiny{$\rightarrow I$}}
   		\UnaryInfC{$ A \rightarrow_n B$}
   		\DisplayProof 
   		&   		
   		\AxiomC{$  \bot $}
    		\LeftLabel{\tiny{$\bot$}}
   		\UnaryInfC{$ A$}
   		\DisplayProof
   		\\[6 ex]
   \end{tabular}
   \end{center}

\begin{flushleft}
   \textbf{Formalized Rules:}
   \end{flushleft}
\begin{center}
\begin{tabular}{c c}
   		   		
   		\AxiomC{$A \rightarrow_n B$}
    		\AxiomC{$A \rightarrow_n C$}
    		\LeftLabel{\tiny{$(\wedge I)_f$}}
    		\BinaryInfC{$A \rightarrow_n B \wedge C$}
    		\DisplayProof
    		&
    		\AxiomC{$A \rightarrow_n C$}
    		\AxiomC{$B \rightarrow_n C$}
    		\LeftLabel{\tiny{$(\vee E)_f$}}
    		\BinaryInfC{$A \vee B \rightarrow_n C$}
    		\DisplayProof
    		\\[4 ex]
    		
\end{tabular}
\end{center}
\begin{center}
      \begin{tabular}{c}
  
    		\AxiomC{$A \rightarrow_n B$}
    		\AxiomC{$B \rightarrow_n C$}
    		\LeftLabel{\tiny{$tr_f$}}
    		\BinaryInfC{$A \rightarrow_n C$}
    		\DisplayProof
    		\\[4 ex]
   		
	\end{tabular}
\end{center}
\begin{flushleft}
   \textbf{Structural Rules:}
   \end{flushleft}
   \begin{center}
   \begin{tabular}{c c}
   		\AxiomC{$A$}
   		\AxiomC{[$A$]}
   		\noLine
   		\UnaryInfC{$\mathcal{D}$}
   		\noLine
   		\UnaryInfC{$B$}
   		\LeftLabel{\tiny{$tr$}}
   		\BinaryInfC{$B$}
   		\DisplayProof 
   		&
   		\AxiomC{$ A \rightarrow_n B $}
    		\LeftLabel{\tiny{$H$}}
   		\UnaryInfC{$A \rightarrow_m B$}
   		\DisplayProof
   		
   		\end{tabular}
\end{center}
Moreover, consider the following set of rules:
\begin{center}
  	\begin{tabular}{c c c}
		\AxiomC{$ A$}
		\AxiomC{$ \neg_n A$}
		\LeftLabel{\tiny{$C$}}
		\BinaryInfC{$\bot$}
		\DisplayProof
		&	
		\AxiomC{$A $}
		\AxiomC{$A \rightarrow_n B$}
		\LeftLabel{\tiny{$R$}}
		\BinaryInfC{$ B$}
		\DisplayProof
		\\[3 ex]
		
		\AxiomC{}
		\LeftLabel{\tiny{$D$}}
		\UnaryInfC{$A \vee \neg_n A$}
		\DisplayProof
		&
        \AxiomC{$(A \wedge (A \rightarrow_n B)) \rightarrow_{n+1} B$}
        \LeftLabel{\tiny{$L$}}
        \UnaryInfC{$A \rightarrow_n B$}
        \DisplayProof
		\\[3 ex]
	\end{tabular}
\end{center}
The condition for the rule $H$ is that $m \geq n$ and the condition of applying the rule $\rightarrow I$ is that $n$ should be strictly greater than all the indices occurred in the hypothesis of the deduction, including $A$.\\

The logic $\mathbf{BPC}_h$ is defined as the system consists of the propositional rules, the structural rules and the formalized rules. Then $\mathbf{EBPC}_h$ is defined as $\mathbf{BPC}_h + C$, $\mathbf{FPL}_h$ is defined as $\mathbf{BPC}_h + L$, $\mathbf{IPC}_h$ is defined as $\mathbf{BPC}_h + R$ and finally $\mathbf{CPC}_h$ is defined as $\mathbf{IPC}_h + D$.
\begin{rem}\label{t3-3}
Consider the following rules:
\begin{center}
  	\begin{tabular}{c c}
		
		\AxiomC{$\top \rightarrow_n \bot$}
		\LeftLabel{\tiny{$C'$}}
		\UnaryInfC{$\bot$}
		\DisplayProof
		&	
		\AxiomC{$\top \rightarrow_n A$}
		\LeftLabel{\tiny{$R'$}}
		\UnaryInfC{$ A$}
		\DisplayProof
		\\[3 ex]
		
	\end{tabular}
\end{center}
It is possible to define $\mathbf{EBPC}_h$ as $\mathbf{BPC}_h + D'$ and define $\mathbf{IPC}_h$ as $\mathbf{BPC}_h + R'$. It is clear that $D'$ and $R'$ are special cases of $D$ and $R$, respectively. Therefore it remains to show that $D'$ and $R'$ can simulate $D$ and $R$, respectively. The following proofs show that it is the case:
\begin{center}
  	\begin{tabular}{c c}
	
		\AxiomC{$ A$}
		\UnaryInfC{$\top \rightarrow_n A $}
		\AxiomC{$ A \rightarrow_n \bot$}
		\BinaryInfC{$\top \rightarrow_n \bot$}
		\LeftLabel{\tiny{$C'$}}
		\UnaryInfC{$\bot$}
		\DisplayProof
		&	
		\AxiomC{$ A$}
		\UnaryInfC{$\top \rightarrow_n A $}
		\AxiomC{$A \rightarrow_n B$}
		\BinaryInfC{$\top \rightarrow_n B$}
		\LeftLabel{\tiny{$R'$}}
		\UnaryInfC{$ B$}
		\DisplayProof
		
	\end{tabular}
\end{center}
\end{rem}
\begin{rem}\label{t3-4}
The usual logics $\mathbf{BPC}$, $\mathbf{EBPC}$, $\mathbf{IPC}$ and $\mathbf{CPC}$ are defined just like their counterparts replacing all $\rightarrow_n$ by $\rightarrow$. (See \cite{Vi}, \cite{Ar}.) Moreover, the logic $\mathbf{FPL}$ is defined as $\mathbf{BPC}$ plus the following rule (See \cite{Vi}):
\begin{center}
  	\begin{tabular}{c}	
        \AxiomC{$(\top \rightarrow A) \rightarrow A$}
        \UnaryInfC{$\top \rightarrow A$}
        \DisplayProof
		\\[3 ex]
	\end{tabular}
\end{center}
It is also possible to define $\mathbf{FPL}$ as $\mathbf{BPC}$ plus the rule:
\begin{center}
  	\begin{tabular}{c}	
        \AxiomC{$(A \wedge (A \rightarrow B)) \rightarrow B$}
        \UnaryInfC{$A \rightarrow B$}
        \DisplayProof
		\\[3 ex]
	\end{tabular}
\end{center}
The equivalence of these two definitions, is based on the fact that the first rule is a special case of the second rule when we have $A =\top$ and the following proof which shows that the first is also powerful enough to simulate the second:\\

\begin{center}
\begin{tabular}{c}
        \AxiomC{$ $}
        \doubleLine
        \UnaryInfC{$A \rightarrow \top$}
        
        \AxiomC{$[\top \rightarrow (A \rightarrow B)]^2$}
        
	    \AxiomC{$[A]^1$}
	    \doubleLine
	    \UnaryInfC{$\top \rightarrow A$}
	    
        \BinaryInfC{$\top \rightarrow (A \wedge (A \rightarrow B))$}
        \AxiomC{$[(A \wedge (A \rightarrow B)) \rightarrow B]^3$}
        \BinaryInfC{$\top \rightarrow B$}
        
        \BinaryInfC{$A \rightarrow B$}
        \LeftLabel{\tiny{$\rightarrow I_2$}}
        \UnaryInfC{$(\top \rightarrow (A \rightarrow B)) \rightarrow (A \rightarrow B)$}
        \UnaryInfC{$\top \rightarrow (A \rightarrow B)$}
        \LeftLabel{\tiny{$(*)$}}
        \doubleLine
        \UnaryInfC{$A \rightarrow B$}
        \LeftLabel{\tiny{$\rightarrow I_1$}}
        \UnaryInfC{$A \rightarrow ((A \rightarrow B))$}
        \doubleLine
        \UnaryInfC{$A \rightarrow (A \wedge (A \rightarrow B))$}
        
        \AxiomC{$[(A \wedge (A \rightarrow B)) \rightarrow B]^3$}
        \insertBetweenHyps{\hskip -100pt}
        \BinaryInfC{$(A \rightarrow B)$}
        \DisplayProof
\end{tabular}
\end{center}
Notice that the double lines mean simple sub-proofs that we do not mention and $(*)$ is the sub-proof which proves 
\[
A, (\top \rightarrow (A \rightarrow B)), ((A \wedge (A \rightarrow B)) \rightarrow B) \vdash A \rightarrow B
\] 
\end{rem}
\begin{rem}\label{t3-5}
In the untyped case of logics introduced in Remark \ref{t3-4}, there is no need to add the rule $tr$, simply because it is admissible. It is enough to put a proof for $\Gamma, A \vdash B$ under the proof of $\Gamma' \vdash A$ to have a proof for $\Gamma, \Gamma' \vdash B$. Unfortunately, this is not the case for the typed version. The reason is simple: If we put the proof of $\Gamma' \vdash A$ under the proof of $\Gamma, A \vdash B$ it means that we keep the structure of two proofs, by changing the premises to the set $\Gamma \cup \Gamma'$. But then it is possible that some applications of the rule $\rightarrow I$ become invalid simply because it is possible to have formulas in the set $\Gamma'$ with greater indices. To have an example, consider the following trees:\\

\begin{center}
  	\begin{tabular}{c c}	
  	    \AxiomC{$p \wedge (\top \rightarrow_2 q)$}
        \UnaryInfC{$p$}
  	    \AxiomC{$[p]$}
  	    \LeftLabel{\tiny{$(*)$}}
  	    \UnaryInfC{$\top \rightarrow_1 p$}
  	    \LeftLabel{\tiny{$tr$}}
        \BinaryInfC{$\top \rightarrow_1 p$}
        \DisplayProof
        \; \; \; \; \;
        &	
  	    \AxiomC{$p \wedge (\top \rightarrow_2 q)$}
        \UnaryInfC{$p$}
        \LeftLabel{\tiny{$(**)$}}
  	    \UnaryInfC{$\top \rightarrow_1 p$}
        \DisplayProof
		\\[3 ex]
	\end{tabular}
\end{center}

The left tree is a valid proof while the right one is not. The reason is that in applying the rule $\rightarrow I$ in $(**)$, the index $1$ is not greater than all the indices occurring in the premises which is $2$, in this case. But in the left tree, applying $\rightarrow I$ in $(*)$ is valid because there is no index in the premises which in this case is the empty set.
\end{rem}
In the following we will define the provability interpretation (BHK interpretation) of the statements in the propositional language $\mathcal{L}_{\infty}$:
\begin{dfn}\label{t3-6}
Let $(M, \{T_n\}_{n=0}^{\infty})$ be a provability model and $A \in \mathcal{L}_{\infty}$ a formula. Then by an arithmetical substitution $\sigma$ we mean a function from atomic formulas to the set of arithmetical sentences. Moreover, define $A^{\sigma}$ as follows:
\begin{itemize}
\item[$(i)$]
If $p$ is an atomic formula, $p^{\sigma}=\Pr_0(\sigma(p))$. Moreover, $\top^{\sigma}=\Pr_0(\top)$ and $\bot^{\sigma}=\Pr_0(\bot)$.
\item[$(ii)$]
$(B \circ C)^{\sigma}=B^{\sigma} \circ C^{\sigma}$ for all $\circ \in \{\wedge, \vee\}$.
\item[$(iii)$]
$(A \rightarrow_n B)^{\sigma}=\Pr_n(A^{\sigma} \rightarrow B^{\sigma})$.
\end{itemize}
Moreover, if $\Gamma$ is a sequence of formulas $A_i$, by $\Gamma^{\sigma}$ we mean the sequence of $A_i^{\sigma}$.
\end{dfn}
\begin{dfn}\label{t3-7}
Let $(M, \{T_n\}_{n=0}^{\infty})$ be a provability model and $A \in \mathcal{L}_{\infty}$ a formula. Then we say $(M, \{T_n\}_{n=0}^{\infty}) \vDash A$ if for any arithmetical substitution $\sigma$, $M \vDash A^{\sigma}$. Moreover, if $\Gamma$ and $\Delta$ are sequences of formulas and $\mathcal{C}$ a class of provability models, by $\mathcal{C} \vDash \Gamma \Rightarrow \Delta$, we mean that for any $(M, \{T_n\}_{n=0}^{\infty}) \in \mathcal{C}$, and for any arithmetical substitution $\sigma$, if $M \vDash \bigwedge \Gamma^{\sigma}$, then $M \vDash \bigvee \Delta^{\sigma}$.
\end{dfn}
Let us illuminate this definition by an example:
\begin{exam}\label{t3-8}
Consider the pair $(\mathbb{N}, \{T_n\}_{n=0}^{\infty})$ in which $T_0=\PA$ and $T_{n+1}=T_n + \Rfn(T_n)$. First of all, it is easy to check that this pair is a reflexive provability model. Secondly, we want to show that this model satisfies the statement $(A \wedge (A \rightarrow_n B)) \rightarrow_{n+1}B$. To show this fact, suppose that $\sigma$ is an arbitrary arithmetical substitution, then we have to show
\[
\mathbb{N} \vDash \Box_{n+1}(A^{\sigma} \wedge \Box_n(A^{\sigma} \rightarrow B^{\sigma}) \rightarrow B^{\sigma})
\]
which holds because $T_{n+1}$ has the reflection principle for the theory $T_n$. Specially, $T_{n+1} \vdash \Box_n(A^{\sigma} \rightarrow B^{\sigma}) \rightarrow (A^{\sigma} \rightarrow B^{\sigma})$. 
\end{exam}
In the remaining part of this section we will introduce the G\"{o}del translation $b: \mathcal{L}_{\infty} \to \mathcal{L}_{\infty}^{\Box}$ and we will show its soundness-completeness property.
\begin{dfn}\label{t3-9}
The translation $b: \mathcal{L}_{\infty} \to \mathcal{L}_{\infty}^{\Box}$ is defined as follows:
\begin{description}
\item[$(i)$]
$ p^{b}=\Box_0 p$, $ \top^{b}=\Box_0 \top$ and $\bot^{b}=\Box_0 \bot$.
\item[$(ii)$]
$(A\wedge B)^b= A^{b} \wedge B^{b}$
\item[$(iii)$]
$(A\vee B)^{b}=A^{b}\vee B^{b}$
\item[$(iv)$]
$(A\to_n B)^{b}=\Box_n(A^{b} \to B^{b})$
\end{description}
\end{dfn}
\begin{rem}\label{t3-10}
It is possible to have a similar translation $b': \mathcal{L} \to \mathcal{L}_{\Box} $ which is like the translation $b$ except in the implicational case which is defined as $(A \rightarrow B)^{b'}=\Box (A^{b'} \rightarrow B^{b'})$. Since it is possible to recognize from the context that which translation we are using, we will use $b$ for $b'$ as well.
\end{rem}
First we need the following important lemma which intuitively states that all translated formulas are boxed inherently:
\begin{lem}\label{t3-11}
$\mathbf{K4}_h \vdash A^b \rightarrow \Box_n A^b$.
\end{lem}
\begin{proof}
The proof is by induction on the complexity of $A$. If $A$ is an atom, $\top$ or $\bot$, we have $A^b=\Box_0(A)$. Therefore, by axiom $\mathbf{4}_h$, we have $\mathbf{K4}_h \vdash \Box_0 A \rightarrow \Box_n \Box_0 A$. For conjunction, by IH we have $\mathbf{K4}_h \vdash B^b \rightarrow \Box_n B^b$ and $\mathbf{K4}_h \vdash C^b \rightarrow \Box_n C^b$, therefore, we have $\mathbf{K4}_h \vdash B^b \wedge C^b \rightarrow \Box_n (B^b \wedge C^b)$. For disjunction, by IH, we have $\mathbf{K4}_h \vdash B^b \rightarrow \Box_n B^b$ and $\mathbf{K4}_h \vdash C^b \rightarrow \Box_n C^b$. Therefore, $\mathbf{K4}_h \vdash B^b \vee C^b \rightarrow \Box_n B^b \vee \Box_n C^b$ and since $\mathbf{K4}_h \vdash \Box_n B^b \vee \Box_n C^b \rightarrow \Box_n (B^b \vee C^b)$, we have the claim. For the implication $B \rightarrow_m C$, again by $\mathbf{4}_h$, we have $\mathbf{K4}_h \vdash \Box_m(B^b \rightarrow C^b) \rightarrow \Box_n \Box_m (B^b \rightarrow C^b)$.
\end{proof}
We need the following theorems about the systems $\mathbf{FPL}$, $\mathbf{GL}$, $\mathbf{FPL}_h$ and $\mathbf{GL}_h$. The strategy is reducing the soundness-completeness of the translation $b$ between $\mathbf{FPL}_h$ and $\mathbf{GL}_h$ to the soundness-completeness between $\mathbf{FPL}$ and $\mathbf{GL}$.
\begin{thm}\label{t3-12}\cite{Vi}
$\Gamma \vdash_{\mathbf{FPL}} A$ iff $\Gamma^b \vdash_{\mathbf{GL}} A^b$.
\end{thm}
\begin{dfn}\label{t3-13}
Let $A$ be a usual propositional formula. By a witness $w$ for $A$ we mean an assignment which assigns natural numbers to implications in a way that if $n$ is assigned to the implication in $A \rightarrow B$ then $n$ should be strictly greater than all numbers assigned to the implications in $A$ and $B$. Moreover, by $A(w)$ we mean a formula in $\mathcal{L}_{\infty}$ substituting any occurrence of implication with $\rightarrow_n$ when $n$ is a number which $w$ assigns to that occurrence. Moreover, by the forgetful translation $f: \mathcal{L}_{\infty} \to \mathcal{L}$ we mean a function which translates atomic formulas, conjunctions and disjunctions to themselves and sends $\rightarrow_n$ to $\rightarrow$. Notice that there exists a witness $w$ for $A^f$ such that $A^f(w)=A$.
\end{dfn}
\begin{lem}\label{t3-14}
\begin{itemize}
\item[$(i)$]
For any $A \in \mathcal{L}_{\infty}$ and any natural numbers $m, n$ greater than all the indices in $A$ and $B$, $ A \rightarrow_m B \vdash_{\mathbf{FPL}_h} A \rightarrow_n B$.
\item[$(ii)$]
For any $A \in \mathcal{L}$ and any witnesses $u$ and $v$ for $A$, $ A(u) \vdash_{\mathbf{FPL}_h} A(v)$
\end{itemize}
\end{lem}
\begin{proof}
For $(i)$ it is enough to show that if $n$ is greater than all indices in $A$ and $B$,
$
A \rightarrow_{n+1} B \vdash_{\mathbf{FPL}_h} A \rightarrow_n B
$. We have $A \rightarrow_{n+1} B \vdash_{\mathbf{FPL}_h} (A \wedge (A \rightarrow_n B) \rightarrow_{n+1} B) $. By the rule $L$
\[
(A \wedge (A \rightarrow_n B) \rightarrow_{n+1} B) \vdash_{\mathbf{FPL}_h} A \rightarrow_n B
\]
hence $A \rightarrow_{n+1} B \vdash_{\mathbf{FPL}_h} A \rightarrow_n B$.\\
For $(ii)$. Use induction on $A$. The atomic case and the case for conjunction and disjunction are easy to check. For the implicational case assume $A= B \rightarrow C$. Then we know that $u=(u', n, u'')$ and $v=(v', m, v'')$ such that $n$ is bigger than all numbers in $u'$ and $u''$ and also $m$ is bigger than all numbers in $v'$ and $v''$. Pick $k=max\{m, n\}$. By IH, $B(v') \vdash_{\mathbf{FPL}_h} B(u')$ and $C(u'') \vdash_{\mathbf{FPL}_h} C(v'')$ Therefore, $\mathbf{FPL}_h \vdash B(v') \rightarrow_k B(u')$ and $\mathbf{FPL}_h \vdash C(u'') \rightarrow_k C(v'')$. By $tr_f$ we have
$
B(u') \rightarrow_k C(u'') \vdash_{\mathbf{FPL}_h} B(v') \rightarrow_k C(v'')
$.
On the other hand by $(i)$ we have
$
B(u') \rightarrow_n C(u'') \vdash_{\mathbf{FPL}_h} B(u') \rightarrow_k C(u'')
$
and
$
B(v') \rightarrow_k C(v'') \vdash_{\mathbf{FPL}_h} B(v') \rightarrow_m C(v'')
$
hence
$
B(u') \rightarrow_n C(u'') \vdash_{\mathbf{FPL}_h} B(v') \rightarrow_m C(v'')
$.
\end{proof}
\begin{thm}\label{t3-15}
\begin{itemize}
\item[$(i)$]
If $\Gamma \vdash_{\mathbf{FPL}_h} A$ then $\Gamma^f \vdash_{\mathbf{FPL}} A^f$.
\item[$(ii)$]
Let $\Gamma \cup \{A\} \subseteq \mathcal{L}$ be a set of usual propositional formulas and $w$ and $w'$ are witnesses for $\Gamma$ and $A$ respectively, then if $\Gamma \vdash_{\mathbf{FPL}} A$ then $\Gamma(w) \vdash_{\mathbf{FPL}_h} A(w')$.
\end{itemize}
\end{thm}
\begin{proof}
$(i)$ is clear by the Remark \ref{t3-4}. For $(ii)$, use induction on the length of the proof of $A$. The important cases are the axiom case, the case of $\rightarrow I$ and the case of $tr_f$.\\

1. For the axiom case, we have $A \in \Gamma$. Assume that the witness for $A$ in $\Gamma$ is $u$, then by Theorem \ref{t3-14}, $ A(u) \vdash_{\mathbf{FPL}_h} A(w')$ which completes the proof.\\

2. For the $\rightarrow I$ case, if $w=(u, n, v)$ then by IH, we know that $\Gamma(w), B(u) \vdash_{\mathbf{FPL}_h} C(v)$. Pick $k$ bigger than all numbers in $w$, $u$ and $v$. Then by $\rightarrow I$ in $\mathbf{FPL}_h$ we have
$\Gamma(w) \vdash_{\mathbf{FPL}_h} B(u) \rightarrow_k C(v)$. Now by Theorem \ref{t3-14}, we will have $\Gamma(w) \vdash_{\mathbf{FPL}_h} B(u) \rightarrow_n C(v)$.\\

3. For the $tr$ case, assume that $w'=(u, n, v)$. Pick $y$ as a witness for $C$ and $k$ bigger than all numbers in $w$, $w'$ and $y$. By IH, we have $\Gamma(w) \vdash_{\mathbf{FPL}_h} B(u) \rightarrow_k C(y)$ and $\Gamma(w) \vdash_{\mathbf{FPL}} C(y) \rightarrow_k D(v)$, then by $tr_f$ in $\mathbf{FPL}_h$ we have $\Gamma(w) \vdash_{\mathbf{FPL}_h} B(u) \rightarrow_k D(v)$. Finally by using the Theorem \ref{t3-14} part $(i)$, $\Gamma(w) \vdash_{\mathbf{FPL}_h} B(u) \rightarrow_n D(v)$. 
\end{proof}
It is time to prove the soundness-completeness of the translation $b$.
\begin{thm}\label{t3-16}(Soundness-completeness of $b$)
\begin{itemize}
\item[$(i)$]
$\Gamma \vdash_{\mathbf{BPC}_h} A$ iff $\Gamma^b \vdash_{\mathbf{K4}_h} A^b$.
\item[$(ii)$]
$\Gamma \vdash_{\mathbf{EBPC}_h} A$ iff $\Gamma^b \vdash_{\mathbf{KD4}_h} A^b$.
\item[$(iii)$]
$\Gamma \vdash_{\mathbf{IPC}_h} A$ iff $\Gamma^b \vdash_{\mathbf{S4}_h} A^b$.
\item[$(iv)$]
$\Gamma \vdash_{\mathbf{FPL}_h} A$ iff $\Gamma^b \vdash_{\mathbf{GL}_h} A^b$.
\end{itemize}
\end{thm}
\begin{proof}\textit{(Soundness of the translation $b$).}
The proof is by induction on the length of the proofs in the propositional logics. More precisely, if we denote the propositional logic by $L$ and its modal counterpart by $L_{\Box}$, then  we will show that if $\Gamma \vdash_{L} A$ then $L_{\Box} \vdash \bigwedge \Gamma^b \rightarrow A^b$.\\

1. For the rules $\wedge I$, $\wedge E$, $\vee I$ and $\vee E$, the claim is easy to check. It is just an easy consequence of the fact that the translation $b$ commutes with conjunctions and disjunctions.\\

The following cases for the rules $\rightarrow I$, $\bot$ and the formalized rules are shown for $\mathbf{K4}_h$ but the proof for the other modal logics are the same.\\

2. For the rule $\rightarrow I$, by IH, we have $L_{\Box} \vdash \bigwedge \Gamma^b \wedge A^b \rightarrow B^b$. Therefore, $L_{\Box} \vdash \bigwedge \Gamma^b \rightarrow (A^b \rightarrow B^b) $. Since $n$ is bigger than all the indices in the statement $\bigwedge \Gamma^b \rightarrow (A^b \rightarrow B^b)$, by necessitation, we have 
\[
L_{\Box} \vdash \Box_n(\bigwedge \Gamma^b \rightarrow (A^b \rightarrow B^b))
\]
By the use of the axiom $\mathbf{K}_h$, we have
\[
L_{\Box} \vdash \Box_n(\bigwedge \Gamma^b) \rightarrow \Box_n(A^b \rightarrow B^b)
\]
By Lemma \ref{t3-11}, we have
\[
L_{\Box} \vdash \bigwedge \Gamma^b \rightarrow \Box_n(\bigwedge \Gamma^b) \;\;\; (*)
\]
therefore, by $(*)$ and the definition of $b$,
\[
L_{\Box} \vdash \bigwedge \Gamma^b \rightarrow (A \rightarrow_n B)^b
\]

3. For the rule $\bot$, by induction on $A$ we will show that $L_{\Box} \vdash \Box_0 \bot \rightarrow A^b$. For the atomic case, we have 
$L_{\Box} \vdash \bot \rightarrow p$, therefore by necessitation and the axiom $\mathbf{K}_h$ we have $L_{\Box} \vdash \Box_0 \bot \rightarrow \Box_0 p$ which is what we wanted. The conjunction and dijunction cases are easy to check. For the implication, we have $L_{\Box} \vdash \bot \rightarrow (A^b \rightarrow B^b)$. Hence, $L_{\Box} \vdash \Box_n \bot \rightarrow \Box_n(A^b \rightarrow B^b)$. Since $0 \leq n$, $L_{\Box} \vdash \Box_0 \bot \rightarrow \Box_n \bot$, thus $L_{\Box} \vdash \Box_0 \bot \rightarrow (A \rightarrow_n B)^b$.\\

For the rule $tr$, by IH if $L_{\Box} \vdash \bigwedge \Gamma^{b} \rightarrow A^b$ and $L_{\Box} \vdash \bigwedge \Gamma'^{b} \wedge A^b \rightarrow B^b$, then we obviously have $L_{\Box} \vdash \bigwedge \Gamma^{b} \wedge \bigwedge \Gamma'^b \rightarrow B^b$.\\ 

4. The case of formalized $\land I$, formalized $\vee E$ and $tr$. For the formalized $\wedge I$, by IH, we have
\[
L_{\Box} \vdash \bigwedge \Gamma^b \rightarrow \Box_n(A^b \rightarrow B^b)
\]
and
\[
L_{\Box} \vdash \bigwedge \Gamma^b \rightarrow \Box_n(A^b \rightarrow C^b)
\]
therefore, by some applications of the axiom $\mathbf{K}_h$ and the modus ponens rule, we have
\[
L_{\Box} \vdash \bigwedge \Gamma^b \rightarrow \Box_n(A^b \rightarrow B^b \land C^b)
\]
which completes the proof. The case for formalized $\vee E$ and $tr$ are the same.\\

5. For the rule $C$, by IH, we have
\[
\mathbf{KD4}_h \vdash \bigwedge \Gamma^b \rightarrow \Box_n(A^b \rightarrow \Box_0 \bot) \wedge A^b
\]
By Lemma \ref{t3-11}, $\mathbf{K4}_h \vdash A^b \rightarrow \Box_n A^b $
and
\[
\mathbf{KD4}_h \vdash \Box_n(\Box_0 \bot \leftrightarrow \bot)
\]
we have
\[
\mathbf{KD4}_h \vdash \bigwedge \Gamma^b \rightarrow \Box_n \bot
\]
and since $\mathbf{KD4}_h \vdash \Box_n \bot \rightarrow \bot$ and $\mathbf{KD4}_h \vdash \bot \rightarrow \Box_0 \bot$
then
\[
\mathbf{KD4}_h \vdash \bigwedge \Gamma^b \rightarrow \Box_0 \bot
\]

For the rule $R$ by IH, we have
\[
\mathbf{S4}_h \vdash \bigwedge \Gamma^b \rightarrow \Box_n(A^b \rightarrow B^b) \wedge A^b
\]
By Lemma \ref{t3-11}, $\mathbf{K4}_h \vdash A^b \rightarrow \Box_n A^b $,
which means
\[
\mathbf{S4}_h \vdash \bigwedge \Gamma^b \rightarrow \Box_n B^b
\]
we know that $\mathbf{S4} \vdash \Box_n B^b \rightarrow B^b$, hence
\[
\mathbf{S4}_h \vdash \bigwedge \Gamma^b \rightarrow B^b
\]

And finally for the rule, $L$, by IH
\[
\mathbf{GL}_h \vdash \Box_{n+1}(A^b \wedge \Box_n(A \rightarrow B) \rightarrow B^b)
\]
which means
\[
\mathbf{GL}_h \vdash \Box_{n+1}(\Box_n (A^b \rightarrow B^b) \rightarrow (A^b \rightarrow B^b))
\]
therefore
\[
\mathbf{GL}_h \vdash \Box_{n}(A^b \rightarrow B^b)
\]
hence $\mathbf{GL}_h \vdash (A \rightarrow_n B)^b$.
\end{proof}
\begin{proof}\textit{(Completeness of the translation $b$.)}
For the completeness part for the logics $\mathbf{BPC}_h$, $\mathbf{EBPC}_h$ and $\mathbf{IPC}_h$ we will use the sequent calculi for $\mathbf{K4}_h$, $\mathbf{KD4}_h$ and $\mathbf{S4}_h$ as introduced in the Preliminaries. Moreover, the structure of the proof for all of these logics are the same. Therefore, we will prove their completeness theorems simultaneously. To achieve this goal, denote the propositional logic by $L$, its modal counterpart by $L_{\Box}$ and the sequent calculi for $L_{\Box}$ by $G(L_{\Box})$. Assume that $\Gamma^b \vdash_{L_{\Box}} A^b$. Then $G(L_{\Box}) \vdash \Gamma^b \Rightarrow A^b$. Hence, there is a cut-free proof for $\Gamma^b \Rightarrow A^b$. Call it $\pi$. It is obvious that all the formulas occurring in $\pi$ are sub-formulas of $A^b$ or sub-formulas of the elements of $\Gamma^b$. We know that all of these sub-formulas have the following forms: $B^b$; $B^b \rightarrow C^b$ and atoms $p$. ($\top$ and $\bot$ are considered as atomic in this proof.) Therefore, every sequent in $\pi$ has the following form:
\[
\Gamma^b, \{B^b_i \rightarrow C^b_i\}_{i \in I}, \{p_j\}_{j \in J} \Rightarrow \Delta^b, \{D^b_r \rightarrow E^b_r\}_{r \in R}, \{q_s\}_{s \in S} 
\]
We will prove the following claim:\\
 
\textbf{Claim.} If 
\[
G(L_{\Box}) \vdash \; \Gamma^b, \{B^b_i \rightarrow C^b_i\}_{i \in I}, \{p_j\}_{j \in J} \Rightarrow \Delta^b, \{D^b_r \rightarrow E^b_r\}_{r \in R}, \{q_s\}_{s \in S} 
\]
then for any $X \subseteq I$
\[
\Gamma, \{p_j\}_{j \in J}, \{D_r\}_{r \in R}, \{C_i\}_{i \in X} \vdash_L \bigvee \{\Delta, \{q_s\}_{s \in S}, \{E_r\}_{r \in R}, \{B_i\}_{i \notin X}\}
\]

The proof is by induction on the length of the cut-free proof in $G(L_{\Box})$. The case for axioms and structural rules are easy to check. If the last rule is a conjunction or a disjunction rule, then the main formula is in the first form. Then since it is possible to simulate all conjunction and disjunction rules in $\mathbf{BPC}_h$, the case of conjunction and disjunction rules are also easy to check. If the last rule is an implication rule, since we define our claim up to using implicational rules, there is nothing to prove in this case. Finally, if the last rule is a modal rule, then, we have different cases:\\

1. The case $L=\mathbf{K4}_h$. If the last rule is a modal rule, based on the form of the formulas and the fact that in those three forms a boxed formula should have the first form, we have two cases. The first case is when the boxed formula in the right side has the form $\Box_n(D^b \rightarrow E^b)$. The second case is when the formula has the form $\Box_0 p$. For the first case, the last rule has the following form:
\begin{center}
  	\begin{tabular}{c c c}
		\AxiomC{$ \{F^b_r \rightarrow G^b_r\}_{r \in R}, \{p_j, \Box_0 p_j \}_{j \in J}, \{ B_i^b \rightarrow C_i^b, \Box_{n_i}(B_i^b \rightarrow C_i^b) \}_{i \in I} \Rightarrow D^b \rightarrow E^b$}
		\UnaryInfC{$ \Box_n(\{F^b_r \rightarrow G^b_r)\}_{r \in R}, \{\Box_0 p_j \}_{j \in J}, \{ \Box_{n_i}(B_i^b \rightarrow C_i^b) \}_{i \in I} \Rightarrow \Box_n (D^b \rightarrow E^b)$}
		\DisplayProof
		\end{tabular}
\end{center}
By IH and for any $X \subseteq I$ and $Y \subseteq R$ we have
\[
\{G_r\}_{r \in Y}, \{p_j \}_{j \in J}, \{ B_i \rightarrow_{n_i} C_i \}_{i \in I}, \{C_i\}_{i \in X}, D \vdash_{\mathbf{BPC}_h} \{F_r\}_{r \notin Y}, \{B_i\}_{i \notin X}, E
\]
and we want to prove
\[
\{F_r \rightarrow_n G_r\}_{r \in R}, \{p_j \}_{j \in J}, \{ B_i \rightarrow_{n_i} C_i \}_{i \in I} \vdash_{\mathbf{BPC}_h} D \rightarrow_n E.
\]
First notice that we know $n>n_i$, therefore by the rule $\rightarrow I$ the following is provable by $\Sigma=\{p_j \}_{j \in J} \cup \{ B_i \rightarrow_{n_i} C_i \}_{i \in I}$
\[
\bigwedge \{G_r\}_{r \in Y} \wedge \bigwedge \{C_i\}_{i \in X} \wedge D \rightarrow_n \bigvee \{F_r\}_{r \notin Y} \vee \bigvee \{B_i\}_{i \notin X} \vee E.
\]
Fix $i \in I$ and also fix some $Z \subseteq I-\{i\}$. Both of the following statements are theorems of $\Sigma$:
\[
\bigwedge \{G_r\}_{r \in Y} \wedge \bigwedge \{C_i\}_{i \in Z} \wedge D \rightarrow_n \bigvee \{F_r\}_{r \notin Y} \vee \bigvee \{B_i\}_{i \notin Z} \vee \mathbf{B_i} \vee E
\]
and
\[
\bigwedge \{G_r\}_{r \in Y} \wedge \bigwedge \{C_i\}_{i \in Z} \wedge \mathbf{C_i} \wedge D \rightarrow_n \bigvee \{F_r\}_{r \notin Y} \vee \bigvee \{B_i\}_{i \notin Z} \vee E.
\]
Since $\Sigma \vdash B_i \rightarrow_{n_i} C_i$ by using the rule $H$ and the fact that $n_i<n$, we will have $\Sigma \vdash \mathbf{B_i} \rightarrow_n \mathbf{C_i}$. Then by using appropriate formalized rules we have
\[
\bigwedge \{G_r\}_{r \in Y} \wedge \bigwedge \{C_i\}_{i \in Z} \wedge D \rightarrow_n \bigvee \{F_r\}_{r \notin Y} \vee \bigvee \{B_i\}_{i \notin Z}\vee E
\]
provable by $\Sigma$ in $\mathbf{BPC}_h$.
By iterating this method we can eliminate all the elements in $I$ and finally for any $Y \subseteq R$, we have
\[
\Sigma \vdash_{\mathbf{BPC}_h} \bigwedge \{G_r\}_{r \in Y} \wedge D \rightarrow_n \bigvee \{F_r\}_{r \notin Y} \vee E.
\]
Define $\Sigma'=\Sigma + \{F_r \rightarrow_n G_r\}_{r \in R}$. Therefore for any any $Y \subseteq R $
\[
\Sigma' \vdash_{\mathbf{BPC}_h} \bigwedge \{G_r\}_{r \in Y} \wedge D \rightarrow_n \bigvee \{F_r\}_{r \notin Y} \vee E
\]
but $\Sigma' \vdash_{\mathbf{BPC}_h} F_r \rightarrow_n G_r$, hence by the same method as above we can eliminate $R$ and hence we will have
\[
\Sigma' \vdash D \rightarrow_n  E
\]
which is what we wanted to prove.\\

If the boxed formula in the right side of the rule is $\Box_0 p$, since $0$ is the lowest possible index, the rule has the following form:
\begin{center}
  	\begin{tabular}{c c c}
		\AxiomC{$\{p_j \}_{j \in J} \Rightarrow p$}
		\UnaryInfC{$ \{\Box_0 p_j \}_{j \in J} \Rightarrow \Box_0 p$}
		\DisplayProof
		\end{tabular}
\end{center}
therefore, by IH, $\{p_j \}_{j \in J} \vdash_{\mathbf{BPC}_h} p$, which is what we wanted.\\

2. The case $L=\mathbf{KD4}_h$. If the last rule is $\Box_{4_h}R$, the proof is the same as the case 1. If the last rule is $\Box_{D_h}R$, then everything in the proof is the same as the proof for case 1 when we put $D=\top$ and $E=\bot$. Therefore, we have
\[
\{p_j \}_{j \in J}, \{ B_i \rightarrow_{n_i} C_i \}_{i \in I}, \{F_r \rightarrow_n G_r\}_{r \in R} \vdash_{\mathbf{EBPC}_h} \top \rightarrow_n  \bot.
\] 
Then by the rule $C$, we will have
\[
\{p_j \}_{j \in J}, \{ B_i \rightarrow_{n_i} C_i \}_{i \in I}, \{F_r \rightarrow_n G_r\}_{r \in R} \vdash_{\mathbf{EBPC}_h} \bot
\]
which is what we wanted.\\

3. The case $L=\mathbf{S4}_h$. If the last rule is $\Box_{S_h}R$, then the proof is similar to the case 1. If the last rule is $\Box_h L$, then there are two cases. First, the case in which the boxed formula has the form $\Box_n (B^b \rightarrow C^b)$. And the second case in which the boxed formula has the form $\Box_0 p$. For the first case, the rule should have the following form:
\begin{center}
  	\begin{tabular}{c c c}
		\AxiomC{$\Gamma^b, \{B^b_i \rightarrow C^b_i\}_{i \in I}, \{p_j\}_{j \in J}, B^b \rightarrow C^b \Rightarrow \Delta^b, \{D^b_r \rightarrow E^b_r\}_{r \in R}, \{q_s\}_{s \in S}$}
		\UnaryInfC{$\Gamma^b, \{B^b_i \rightarrow C^b_i\}_{i \in I}, \{p_j\}_{j \in J}, \Box_m(B^b \rightarrow C^b) \Rightarrow \Delta^b, \{D^b_r \rightarrow E^b_r\}_{r \in R}, \{q_s\}_{s \in S} $}
		\DisplayProof
		\end{tabular}
\end{center}
Therefore by IH, for any $X \subseteq I$ and $Y \subseteq R$ we have
\[
\Gamma, \{p_j\}_{j \in J}, \{D_r\}_{r \in R}, \{C_i\}_{i \in X}, C \vdash_{\mathbf{IPC}_h} \bigvee \{\Delta, \{q_s\}_{s \in S}, \{E_r\}_{r \in R}, \{B_i\}_{i \notin X}\}
\]
and
\[
\Gamma, \{p_j\}_{j \in J}, \{D_r\}_{r \in R}, \{C_i\}_{i \in X} \vdash_{\mathbf{IPC}_h} B \vee \bigvee \{\Delta, \{q_s\}_{s \in S}, \{E_r\}_{r \in R}, \{B_i\}_{i \notin X}\}.
\]
Since
\[
\Gamma, \{p_j\}_{j \in J}, \{D_r\}_{r \in R}, \{C_i\}_{i \in X}, B \rightarrow_m C \vdash_{\mathbf{IPC}_h} B \rightarrow_m C
\]
and we have $B \rightarrow_m C, B \vdash_{\mathbf{IPC}_h} C$, then
\[
\Gamma, \{p_j\}_{j \in J}, \{D_r\}_{r \in R}, \{C_i\}_{i \in X} \vdash_{\mathbf{IPC}_h} C \vee \bigvee \{\Delta, \{q_s\}_{s \in S}, \{E_r\}_{r \in R}, \{B_i\}_{i \notin X}\}.
\]
By using some appropriate formalized rule and $tr$ on $C$ we will have
\[
\Gamma, \{p_j\}_{j \in J}, \{D_r\}_{r \in R}, \{C_i\}_{i \in X}, B \rightarrow_m C \vdash_{\mathbf{IPC}_h} \bigvee \{\Delta, \{q_s\}_{s \in S}, \{E_r\}_{r \in R}, \{B_i\}_{i \notin X}\}
\]
which is what we wanted. The second case is straightforward by IH.\\

After proving the claim, the theorem is an easy consequence: Since there is a proof of $\Gamma^b \Rightarrow A^b$ in $G(L_{\Box})$, then by claim we have $\Gamma \vdash_L A$. \\

For the logic $\mathbf{GL}_h$, if $\Gamma^b \vdash_{\mathbf{GL}_h} A^b$ then by using the forgetful translation, which forgets the indices of the boxes, we will have $(\Gamma^b)^f \vdash_{\mathbf{GL}} (A^b)^f$. Since for any formula $X \in \mathcal{L}_{\infty}$, $(X^b)^f=(X^f)^b$, hence $(\Gamma^f)^b \vdash_{\mathbf{GL}} (A^f)^b$. Therefore by completeness of $b$ between $\mathbf{GL}$ and $\mathbf{FPL}$, we have $\Gamma^f \vdash_{\mathbf{FPL}} A^f$. Define $w_{\Gamma}$ and $w_A$ such that $\Gamma^f(w_{\Gamma})=\Gamma$ and $A^f(w_A)=A$. Then by Theorem \ref{t3-15} we have $\Gamma^f(w_{\Gamma}) \vdash_{\mathbf{FPL}_h} A^f(w_A)$ which completes the proof.
\end{proof}
And the final part of this section contains the proof of the fact that these propositional logics have the disjunction property, as we expect for any constructive logic:
\begin{thm}\label{t3-17}
All of the logics $\mathbf{BPC}_h$, $\mathbf{EBPC}_h$, $\mathbf{IPC}_h$ and $\mathbf{FPL}_h$ have disjunction property.
\end{thm}
\begin{proof}
The proof for all of these logics are the same. Assume $L$ is one of the mentioned propositional logics, and $L_{\Box}$ is its modal counterpart. Then if $L \vdash A \vee B$ then by soundness of the translation $b$, $L_{\Box} \vdash A^b \vee B^b$. Then, by Lemma \ref{t3-11}, we have $L_{\Box} \vdash \Box_n A^b \vee \Box_m B^b$ for some big enough $m$ and $n$. Then by strong disjunction property for $L_{\Box}$, Theorem \ref{t2-8}, we have: $L_{\Box} \vdash A^b$ or $ L_{\Box} \vdash B^b$. Therefore, by completeness of the translation $b$, we will have $L \vdash A$ or $L \vdash B$. 
\end{proof}
\section{Soundness-Completeness Theorems}
In this section we will prove the soundness-completeness theorems for propositional logics that we introduced in the previous section. To do so, we have to define the notion of a BHK model. It is clear that formalizing the BHK interpretation needs formalizing two different kinds of conditions: The first is the way that the BHK interpretation interprets propositional connectives and the second is the consistency assumption which states that there is no proof for inconsistency. For the first one, we defined the satisfaction relation between provability models and propositional formulas exactly as what the BHK interpretation demands. For the second condition, we need the following discussion: First of all, it seems clear that the natural formalization of this condition is the consistency assumption on the provability model which states that $M \vDash \neg \Pr_n(\bot)$ for all $n \geq 0$. But we have to notice that in the intuitionistic tradition everything should be also reflected in the level of provability. To implement this idea, there are two possible natural ways: The first one is assuming that the meta-theory of $T_n$ is strong enough to prove its consistency, i.e. $M \vDash \Pr_{n+1}(\neg \Pr_n(\bot))$. This statement seems totally natural to assume, but the cost is a lot. In fact, it makes the BHK interpretation much more limited than what we expected. For instance, on the one hand, the statement $(\top \rightarrow_{n} \bot) \rightarrow_{n+1} \bot$ would be valid in all the BHK models which means that $\mathbf{BHK} \vDash (\top \rightarrow_n \bot) \rightarrow_{n+1} \bot$ and on the other hand, $(\top \rightarrow_n \bot) \rightarrow_{n+1} \bot$ is the essential axiom of $\mathbf{EBPC}_h$. Therefore, there is no BHK characterization of logics below $\mathbf{EBPC}_h$ which is not what we expected. The second approach is based on assuming a weaker version which states that $T_{n+1}$ can not prove the inconsistency of $T_n$, i.e $M \vDash \neg \Pr_{n+1}(\Pr_n(\bot))$ for all $n \geq 0$. This formalization seems more liberal than the first one, and we will choose it as our formalization. Notice that for the logics above $\mathbf{EBPC}_h$, this condition is not needed, since the provability models are strong enough to satisfy it automatically.
\begin{dfn}\label{t4-1}
A provability model $(M, \{T_n\}_{n=0}^{\infty})$ is called a BHK model if for all $n \geq 0$, $M \vDash \neg \Pr_{n+1}(\Pr_n(\bot))$. We will denote the class of all BHK models by $\mathbf{BHK}$ and the class of all constant BHK models by $\mathbf{cBHK}$.
\end{dfn}
\begin{thm}\label{t4-2}(Soundness-Completeness)
\begin{itemize}
\item[$(i)$]
$\Gamma \vdash_{\mathbf{BPC}_h} A$ iff $\mathbf{PrM} \vDash \Gamma \Rightarrow A$. Moreover, $\mathbf{BPC}_h \vdash A$ iff $\mathbf{BHK} \vDash A$.
\item[$(ii)$]
$\Gamma \vdash_{\mathbf{EBPC}_h} A$ iff $\mathbf{Cons} \vDash \Gamma \Rightarrow A$. 
\item[$(iii)$]
$\Gamma \vdash_{\mathbf{IPC}_h} A$ iff $\mathbf{Ref} \vDash \Gamma \Rightarrow A$.
\item[$(iv)$]
$\Gamma \vdash_{\mathbf{FPL}_h} A$ iff $\mathbf{Cst} \vDash \Gamma \Rightarrow A$. Moreover, $\mathbf{FPL}_h \vdash A$ iff $\mathbf{cBHK} \vDash A$.
\item[$(v)$]
There is no BHK model $(M, \{T_n\}_{n=0}^{\infty})$ such that $(M, \{T_n\}_{n=0}^{\infty}) \vDash \mathbf{CPC}_h$.
\end{itemize}
\end{thm}
\begin{proof}
First of all notice that for any formula $A \in \mathcal{L}_{\infty}$ and any provability model $(M, \{T_n\}_{n=0}^{\infty})$, $(M, \{T_n\}_{n=0}^{\infty}) \vDash A$ is equivalent to $(M, \{T_n\}_{n=0}^{\infty}) \vDash A^b$ by definition. Therefore $(i)$, $(ii)$, $(iii)$ and $(iv)$ are easy consequences of the Theorem \ref{t3-16} and the soundness-completeness theorem of the corresponding modal logics (Theorem \ref{t2-10}). The remaining part is the completeness theorem for the classes $\mathbf{BHK}$ and $\mathbf{cBHK}$ for $\mathbf{BPC}_h$ and $\mathbf{FPL}_h$, respectively. For $\mathbf{BPC}_h$, define $\Delta$ as the set consisting of all instances of the formula $\neg \Box_{n+1} \Box_n \bot$ for any $n \geq 0$. Notice that $(M, \{T_n\}_{n=0}^{\infty}) \vDash \Delta$, iff $(M, \{T_n\}_{n=0}^{\infty})$ is a BHK model. Since $(M, \{T_n\}_{n=0}^{\infty}) \vDash A$ is equivalent to $(M, \{T_n\}_{n=0}^{\infty}) \vDash A^b$, and since $\mathbf{BHK} \vDash A$, hence $\mathbf{PrM} \vDash \Delta \Rightarrow A^b$. By completeness of $\mathbf{K4}_h$, we have $ \Delta \vdash_{\mathbf{K4}_h} A^b$. Therefore, there are finite instances of the formulas in $\Delta$ such that $ \{\neg \Box_{n_i+1} \Box_{n_i} \bot\}_{i=0}^{r} \vdash_{\mathbf{K4}_h} A^b$. Hence,
\[
 \mathbf{K4}_h \vdash (\bigvee_{i=0}^{r} \Box_{n_i+1} \Box_{n_i} \bot \vee A^b).
\]
Since by the axiom $\mathbf{H}$ we can increase $n_i$'s, w.l.o.g assume $n_i>0$. Moreover we have $\mathbf{K4}_h \vdash \Box_{n_i+1} \Box_{n_i} \bot \rightarrow \Box_{n_i+1} \Box_{n_i} \Box_0 \bot$ and also we know that $(\top \rightarrow_{n_i+1} (\top \rightarrow_{n_i} \bot))^b$ is equivalent to $\Box_{n_i+1} \Box_{n_i} \Box_0 \bot$ provably in $\mathbf{K4}_h$, hence
\[
 \mathbf{K4}_h \vdash (\bigvee_{i=0}^{r} (\top \rightarrow_{n_i+1} (\top \rightarrow_{n_i} \bot)) \vee A)^b
\]
and by completeness of the translation $b$, Theorem \ref{t3-16}, we have
\[
\mathbf{BPC}_h \vdash (\bigvee_{i=0}^{r} (\top \rightarrow_{n_i+1} (\top \rightarrow_{n_i} \bot)) \vee A).
\]
By the disjunction property of $\mathbf{BPC}_h$, Theorem \ref{t3-17}, we have $\mathbf{BPC}_h \vdash A$ or for some $i$, $\mathbf{BPC}_h \vdash (\top \rightarrow_{n_i+1} (\top \rightarrow_{n_i} \bot))$. The latter is impossible because if $\mathbf{BPC}_h \vdash (\top \rightarrow_{n_i+1} (\top \rightarrow_{n_i} \bot))$ then by the soundness of $\mathbf{BPC}$, the formula $(\top \rightarrow_{n_i+1} (\top \rightarrow_{n_i} \bot))$ should be true in the provability model $(\mathbb{N}, \{I\Sigma_1\}_{n=0}^{\infty})$. It means that $\mathbb{N} \vDash \Pr_{n_i+1}(\Pr_{n_i}(\Pr_0(\bot)))$. Hence, $I\Sigma_1 \vdash \Pr_{I\Sigma_1}(\Pr_{I\Sigma_1}(\bot))$ which also means that $I\Sigma_1 \vdash \bot$ which is not the case. Hence, $\mathbf{BPC}_h \vdash A$.
The case for $\mathbf{FPL}_h$ is exactly the same.\\

For $(v)$, we will prove the claim by contradiction. Firstly, we want to show that the following two statements hold:
\begin{itemize}
\item[$(i)$]
$M$ thinks that $T_{2} \vdash \Pr_{1}(\Pr_0(\bot)) \rightarrow \Pr_0(\bot)$ (a weak version of the provability of the consistency assumption).
\item[$(ii)$]
For any arithmetical statement $\phi$, $M$ thinks
\[
\neg \Pr_0(\phi) \rightarrow \Pr_{1}(\Pr_0(\phi) \rightarrow \Pr_0(\bot))
\]
(a weak version of the axiom $\mathbf{5}_h$).
\end{itemize}
For $(i)$, consider the formula $ (\top \rightarrow_1 \bot) \rightarrow_{2} \bot$. Since it is a theorem of $\mathbf{CPC}_h$, then
\[
(M, \{T_n\}_{n=0}^{\infty}) \vDash \Pr_{2} (\Pr_{1} (\Pr_0 (\bot)) \rightarrow \Pr_0 (\bot))
\]
and therefore, we will have $(i)$.\\

Secondly, we know $\mathbf{CPC}_h \vdash p \vee \neg_1 p$. Hence, for any arithmetical substitution $\sigma$ we have $(M, \{T_n\}_{n=0}^{\infty}) \vDash \Pr_0 (p^{\sigma}) \vee \Pr_1 (\Pr_0 (p^{\sigma}) \rightarrow \Pr_0 (\bot))$. Therefore, if we assume $p^{\sigma}=\phi$, then $(ii)$ follows.\\

Using these two statements, we will reach the contradiction. First of all, to simplify the proof, we will use $\Pr(A)$ for $\Pr_{S}$ in which $S=T_1+\Cons(T_0)$. Put $\phi=\Pr(\bot)$, then $(ii)$ would be equivalent to
\[
M \vDash \neg \Pr_0(\Pr(\bot)) \rightarrow \Pr(\neg \Pr_0(\Pr(\bot))).
\]
On other hand by the formalized $\Sigma_1$-completeness, we have 
\[
I\Sigma_1 \vdash \neg \Pr_0(\Pr(\bot)) \rightarrow \neg \Pr(\bot),
\]
hence,
\[
S \vdash \neg \Pr_0(\Pr(\bot)) \rightarrow \neg \Pr(\bot).
\]
Moreover, by $\Sigma_1$-completeness, we have 
\[
I\Sigma_1 \vdash \Pr(\neg \Pr_0(\Pr(\bot)) \rightarrow \neg \Pr(\bot)).
\]
Therefore,
\[
I\Sigma_1 \vdash \Pr(\neg \Pr_0(\Pr(\bot))) \rightarrow \Pr(\neg \Pr(\bot)).
\]
And since $M \vDash I\Sigma_1$, we have
\[
M \vDash \neg \Pr_0(\Pr(\bot)) \rightarrow \Pr(\neg \Pr(\bot)).
\]
Based on G\"{o}del's second incompleteness theorem formalized in $I\Sigma_1$, we can conclude
\[
I\Sigma_1 \vdash \neg \Pr(\bot) \rightarrow \neg \Pr(\neg \Pr(\bot)).
\]
On the other hand, since $(M, \{T_n\}_{n=0}^{\infty})$ is a BHK model, we have 
$
M \vDash \neg \Pr_{1}(\Pr_0(\bot)),
$
hence $M \vDash \neg \Pr(\bot)$. Since $M \vDash I\Sigma_1$, 
$
M \vDash \neg \Pr(\neg \Pr(\perp)).
$
Therefore,
$
M \vDash \Pr_0(\Pr(\bot))
$
and thus by the definition of $S$ we have
$
M \vDash \Pr_0(\Pr_1(\Pr_0(\bot))).
$
By $(i)$, $M \vDash \Pr_2(\Pr_{1}(\Pr_0(\bot)) \rightarrow \Pr_0(\bot))$. Since $M \vDash \Pr_0(\Pr_1(\Pr_0(\bot)))$, we have $M \vDash \Pr_2(\Pr_0(\bot))$. Therefore, we have $M \vDash \Pr_2(\Pr_1(\bot))$ which contradicts with the condition of being a BHK model. Therefore, we reach a contradiction and it proves the theorem.
\end{proof}
\vspace{4pt}
\textbf{Acknowledgment.} We wish to thank Raheleh Jalali and Masoud Memarzadeh for their careful reading of the earlier draft and their useful comments.

\end{document}